\documentclass[12pt]{article}
\usepackage{ae} 
\usepackage[T1]{fontenc}
\usepackage[ansinew]{inputenc}
\usepackage[reqno,namelimits]{amsmath}
\usepackage{amsthm}
\usepackage{amssymb}
\usepackage{graphicx}
\usepackage{fancyheadings}

\def\cvd{~\vbox{\hrule\hbox{%
     \vrule height1.3ex\hskip0.8ex\vrule}\hrule } }

\newtheorem{thm}{Theorem}[section]

\newtheorem*{lemma*}{Lemma}
\newtheorem{cor}[thm]{Corollary}

\newtheorem{prop}[thm]{Proposition}
\newtheorem{defn}[thm]{Definition}
\newtheorem{remark}[thm]{Remark}

\numberwithin{equation}{section}

 \newcommand{\Real}{\mathbb{R}}

 \newcommand{\set}[1]{\left\{#1\right\}}
 
 \newcommand{\norm}[1]{\left\Vert#1\right\Vert}
 
\author{E. Glakousakis and S. Mercourakis}

\title{Examples of infinite dimensional Banach spaces without infinite equilateral sets}
\begin{document}
\maketitle
\thispagestyle{fancy}
\lfoot{\vspace{1.2cm}\small{\textit{2010 Mathematics Subject Classification: Primary 46B20; Secondary 46B04.}}}
\begin{abstract}
An example of an infinite dimensional and separable Banach space is given, that is not isomorphic to a subspace of
$l_1$ with no infinite equilateral sets.
\end{abstract}


\section*{Introduction}

\hspace*{0.5cm} A subset $S$ of a Banach space $X$ is said to be equilateral if there exists a constant $\lambda>0$ such that $\|x-y\|=\lambda$, for $x, y\in S$ with $x\neq y$.

The question whether an infinite dimensional Banach space contains an infinite equilateral set has been answered in the negative by Terenzi in \cite{4}, who constructed an equivalent norm $|\|\cdot|\|$ on $l_1$ such that the Banach space $(l_1, |\|\cdot|\|)$ contains no infinite equilateral sets. We note that quite recently have also been proved some positive results, every renorming of $c_0$ \cite{3} and every uniformly smooth Banach space \cite{2}, admits an infinite equilateral set.

In this note, we produce some further examples of (infinite dimensional) Banach spaces not admitting infinite equilateral sets. To this end, we follow the method of Terenzi’s proof in \cite{4}. (For the sake of completeness, we give a more elaborate description of the involving parts of his arguments.) Actually, we consider a sequence of (real) finite dimensional Banach spaces $(X_n)_{n\geq 1}$ each having an 1-unconditional basis, and we prove, in almost the same way, as Terenzi in his original proof (in this case, $X_n=\Real$, for $n\geq1$) that the space $Z=(\sum_{n=1}^{\infty}\oplus X_n)_1$ has an equivalent (dual) norm $|\|\cdot|\|$ such that the space $(Z, |\|\cdot|\|)$ does not admit any infinite equilateral sets (Theorem \ref{thm2.2}). As an easy consequence of Theorem \ref{thm2.2} and standard results \cite{1} and \cite{6}, we get an example of a separable Banach space not isomorphic to a subspace of $l_1$, having no infinite equilateral sets. The note is divided in two sections. The first section is devoted to the construction and the basic results concerning the norm $|\|\cdot|\|$ and the second section is devoted to the proof of our main result.
\\
\underline{Acknowledgments.} We wish to thank G. Vassiliadis for several valuable discussions during the preparation of the paper.

\section{}
Let $(X_n, \norm\cdot_n)_{n\geq1}$ be a sequence of Banach spaces. By $c_{00}((X_n))$ we denote the vector space $\bigcup_{n=1}^{\infty}\set{(x_1, \ldots, x_k, 0, \ldots): x_k\in X_k, \textrm{for} \,\,k=1, 2, \ldots, n}$ and we shall write $(x_1, \ldots, x_n)$ instead of $(x_1, \ldots, x_n, 0, \ldots)$. We also denote by $Z$ the Banach space $(\sum_{n=1}^{\infty}\oplus X_n)_1$.

We define, by induction, a norm on the vector space $c_{00}((X_n))$ as follows: let $n\in\mathbb{N}$ and $x_k\in X_k$, for $k=1, \ldots, n$ then,
\[
\norm{(x_1, \ldots, x_n)}=\norm{x_1}_1, \,\,\,\textrm{if}\,\,\, n=1
\]
and
\begin{equation*}
\begin{split}
\norm{(x_1, \ldots, x_n)}=\left(1-\frac{1}{n+1}\right)(\norm{x_n}_n+\norm{(x_1, \ldots, x_{n-1})})+ \\
&\!\!\!\!\!\!\!\!\!\!\!\!\!\!\!\!\!\!\!\!\!\!\!\!\!\!\!\!\!\!\!\!\!\!\!\!\!\!\!\!\!\!\!\!
\!\!\!\!\!\!\!\!\!\!\!\!\!\!\!\!
+\frac{1}{n+1}\max\set{\frac{\norm{x_n}_{n}}{n}, \norm{(x_1, \ldots, x_{n-1})}},
\end{split}
\end{equation*}
for $n\geq2$. We now let $(X, \norm\cdot)$ be the completion of the space $(c_{00}((X_n)), \norm\cdot)$.

\begin{prop}\label{prop1.1}
Let $m\in \mathbb{N}$ and $x_n\in X_n$, for $n=1, \ldots, m$. Then
\[
\sum_{n=1}^{m}\frac{1}{2}\norm{x_n}_n\leq\sum_{n=1}^{m}\left(1-\frac{1}{n+1}\right)\norm{x_n}_n
\leq\norm{(x_1, \ldots, x_m)}\leq \sum_{n=1}^{m}\norm{x_n}_n.
\]
\end{prop}
\begin{proof}
It is direct that for any $m\in\mathbb{N}$, $\sum_{n=1}^{m}\frac{1}{2}\norm{x_n}_n\leq\sum_{n=1}^{m}(1-\frac{1}{n+1})\norm{x_n}_n$. It remains to prove that $\sum_{n=1}^{m}(1-\frac{1}{n+1})\norm{x_n}_n\leq\norm{(x_1, \ldots, x_m)}\leq\sum_{n=1}^{m}\norm{x_n}_n$, for any $m\in\mathbb{N}$. Both inequalities are proved inductively. For $m=1$ the first inequality apparently holds. We assume that the inequality holds for some $m\geq1$. Then
\begin{eqnarray*}
\norm{(x_1, \ldots, x_m, x_{m+1})} &  \geq & \left(1-\frac{1}{m+2}\right)\left(\norm{x_{m+1}}_{m+1}+\norm{(x_1, \ldots, x_m)}\right)\\
& & \quad\quad\quad\qquad\qquad\qquad+\frac{1}{m+2}\norm{(x_1, \ldots, x_m)}\\
& = & \left(1-\frac{1}{m+2}\right)\norm{x_{m+1}}_{m+1}+\norm{(x_1, \ldots, x_m)}.
\end{eqnarray*}
By the inductive step we get that,
\[
\norm{(x_1, \ldots, x_m, x_{m+1})}\geq \left(1-\frac{1}{m+2}\right)\norm{x_{m+1}}_{m+1}+
\sum_{n=1}^{m}\left(1-\frac{1}{n+1}\right)\norm{x_n}_n,
\]
as desired.

The second inequality also holds for $m=1$. We now assume that the inequality holds for some $m\geq1$. Then,
\begin{eqnarray*}
\norm{(x_1, \ldots, x_m, x_{m+1})} & \leq & \left(1-\frac{1}{m+2}\right)
\left(\norm{x_{m+1}}_{m+1}+\norm{(x_1, \ldots, x_m)}\right)+ \\
& & \,\,\,\,\,+ \frac{1}{m+2}\frac{\norm{x_{m+1}}_{m+1}}{m+1}+\frac{1}{m+2}\norm{(x_1, \ldots, x_m)}\\
& = & \left(1-\frac{1}{m+2}\right)\norm{x_{m+1}}_{m+1}+\norm{(x_1, \ldots, x_m)}+\\
& & \qquad + \frac{1}{m+2}\frac{\norm{x_{m+1}}_{m+1}}{m+1}\\
& = & \frac{(m+1)^2+1}{(m+1)(m+2)}\norm{x_{m+1}}_{m+1}+\norm{(x_1, \ldots, x_m)} \\
& \leq & \norm{x_{m+1}}_{m+1}+\norm{(x_1, \ldots, x_m)}.
\end{eqnarray*}
As before we get the conclusion by the inductive step.
\end{proof}

\begin{remark}\label{rem1.2}
Should we have that $\sum_{n=1}^{m}\norm{x_n}_n>0$, it can be proved that
\[
\norm{(x_1, \ldots, x_m)}>\sum_{n=1}^{m}\left(1-\frac{1}{n+1}\right)\norm{x_n}_n.
\]
Indeed, let $k$ be the first non zero coordinate of the vector $(x_1, \ldots, x_m)$, then
\begin{eqnarray*}
\norm{(x_1, \ldots, x_m)} & = & \left(1-\frac{1}{k+1}\right)\norm{x_k}_k+\frac{1}{k+1}\frac{\norm{x_k}_k}{k}\\
& > & \left(1-\frac{1}{k+1}\right)\norm{x_k}_k.
\end{eqnarray*}
Now we proceed inductively to finish the proof. \,\,\,\,\cvd
\end{remark}
By Proposition \ref{prop1.1} it can be proved that the spaces $Z$ and $X$ are 2-isomorphic, which is the content of the following proposition.

\begin{prop}\label{prop1.3}
The spaces $Z$ and $X$ are 2-isomorphic.
\end{prop}
\begin{proof}
It follows immediately from Proposition \ref{prop1.1}, since $c_{00}((X_n))$ is a dense subspace in both $Z$ and $X$. (Actually, we defined an equivalent norm on $Z$.)
\end{proof}

\begin{prop}\label{prop1.4}
The sequence $(X_n)_{n\geq1}$ is a monotone Schauder decomposition of each of the spaces $X$ and $Z$, furthermore, it is 1-unconditional decomposition $X$ and $Z$.
\end{prop}
\begin{proof}
It is clear that for every $n\in\mathbb{N}$ the space $X_n$ is a closed subspace of both $X$ and $Z$. Now it is direct from the definition of the norms of $X$ and $Z$ that for every $x=(x_1, \ldots, x_n, \ldots)\in c_{00}((X_n))$, every $(\varepsilon_i)_{i\geq1}\in\set{-1, 1}^{\mathbb{N}}$ and each $m\in\mathbb{N}$ we have
\[
\norm{(\varepsilon_1x_1, \ldots, \varepsilon_mx_m)}= \norm{(x_1, \ldots, x_m)}\leq\norm{(x_1, \ldots, x_{m+1})}\,\,\, \textrm{and}
\]
\[
\norm{(\varepsilon_1x_1, \ldots, \varepsilon_mx_m)}_1= \norm{(x_1, \ldots, x_m)}_1\leq\norm{(x_1, \ldots, x_{m+1})}_1.
\]
So we are done.
\end{proof}
It should be mentioned here that by Proposition \ref{prop1.1} the Schauder decomposition $(X_n)_{n\geq1}$ of the space $X$ is equivalent to itself, considered as a Schauder decomposition of $Z$.
\begin{prop}\label{prop1.5}
i) Let $m\geq1$ and $x_n\in X_n$, for $n=1, \ldots, m$ with $\sum_{n=1}^{m}\norm{x_n}_n>0$. Then there exist $1\geq d_1, \ldots, d_m>0$ such that
\[
\norm{(x_1, \ldots, x_m)}= \sum_{n=1}^{m}d_n\norm{x_n}_n.
\]
ii) Let $x=(x_1, \ldots, x_k, \ldots)\in X$. We assume that there exist $c>0$ and $k\in\mathbb{N}$ such that $\frac{\norm{x_n}_n}{n}<c\leq\norm{(x_1, \ldots, x_k)}$, for every $n>k$. Then
\[
\norm x=\norm{(x_1, \ldots, x_k)}+\sum_{n=k+1}^{\infty}\left(1-\frac{1}{n+1}\right)\norm{x_n}_n.
\]
iii) Let $x=(x_1, \ldots, x_k, \ldots)$ be a non zero vector of $X$, then there exists $k\in\mathbb{N}$ such that
\[
\norm x=\norm{(x_1, \ldots, x_k)}+\sum_{n=k+1}^{\infty}\left(1-\frac{1}{n+1}\right)\norm{x_n}_n.
\]
\end{prop}
\begin{proof}
i) For $m=1$, we have nothing to prove. We assume now that the conclusion holds for some $m\geq1$. Let now $x_n\in X_n$, for $n=1, \ldots, m+1$. By the inductive step there exist $d '_1, \ldots, d '_m>0$ such that
\[
\norm{(x_1, \ldots, x_m)}=\sum_{n=1}^{m}d'_n\norm{x_n}_n.
\]
Now we have that
\[
\frac{\norm{x_{m+1}}_{m+1}}{m+1}\geq\norm{(x_1, \ldots, x_m)} \,\,\,
\textrm{or}
\,\,\,
\frac{\norm{x_{m+1}}_{m+1}}{m+1}\leq\norm{(x_1, \ldots, x_m)}.
\]
By the definition of the norm in the first case we take that
\begin{eqnarray*}
\norm{(x_1, \ldots, x_{m+1})} & = & \left(1-\frac{1}{m+2}\right)\norm{(x_1, \ldots, x_m)}+ \\
& & \qquad\qquad+ \left(1-\frac{1}{m+2}+\frac{1}{(m+1)(m+2)}\right)\norm{x_{m+1}}_{m+1}\\
& = & \sum_{n=1}^{m}\left(1-\frac{1}{m+2}\right)d'_n\norm{x_n}_n+ \\
& & \qquad\qquad+\left(1-\frac{1}{m+2}+\frac{1}{(m+1)(m+2)}\right)\norm{x_{m+1}}_{m+1}.
\end{eqnarray*}
In the second case, the definition of the norm gives that
\begin{eqnarray*}
\norm{(x_1, \ldots, x_{m+1})} & = & \norm{(x_1, \ldots, x_m)}+\left(1-\frac{1}{m+2}\right)\norm{x_{m+1}}_{m+1}\\
& = & \sum_{n=1}^{m}d'_n\norm{x_n}_n + \left(1-\frac{1}{m+2}\right)\norm{x_{m+1}}_{m+1}.
\end{eqnarray*}
ii) Firstly, we observe that by the definition of the norm
\begin{equation}\label{eq1}
\norm{(x_1, \ldots, x_{k+1})}=\norm{(x_1, \ldots, x_k)}+\left(1-\frac{1}{k+2}\right)\norm{x_{k+1}}_{k+1}.
\end{equation}
Further, for any $s>k$ we have that
\[
\norm{(x_1, \ldots, x_s)}\geq\norm{(x_1, \ldots, x_k)}, \,\,\,\textrm{by} \,\,\,\textrm{Proposition}\,\, \ref{prop1.4}
\]
\[
\geq c >\frac{\norm{x_{s+1}}_{s+1}}{s+1}.
\]
Again, the definition of the norm gives that
\begin{equation}\label{eq2}
\norm{(x_1, \ldots, x_{s+1})}=\norm{(x_1, \ldots, x_s)}+\left(1-\frac{1}{s+2}\right)\norm{x_{s+1}}_{s+1}.
\end{equation}
Combining \eqref{eq1} and \eqref{eq2}, we take that for any $s>k$
\[
\norm{(x_1, \ldots, x_{s})}=\norm{(x_1, \ldots, x_k)}+\sum_{n=k+1}^{s}\left(1-\frac{1}{n+1}\right)\norm{x_n}_n.
\]
Now we let $s \to \infty$ to take the conclusion.

iii) Let $x\in X$ with $x\neq0$. Since $x$ also belongs to $Z$, by the definition of the norm of $Z$, there exists $M>0$ such that $\norm{x_n}_n\leq M$, for every $n\in \mathbb{N}$. Now we have that $\frac{M}{n}\to 0$ and  (by Proposition \ref{prop1.4}) $\norm{(x_1, \ldots, x_n)}\to\norm x>0$. So there exists $k\in\mathbb{N}$ such that
\[
\norm{(x_1, \ldots, x_k)}>\frac{M}{k}>\frac{\norm{x_n}_n}{n}, \,\,\,\textrm{for\,\, every}\,\, n>k.
\]
Then by Proposition \ref{prop1.4}, we have that
\[
\norm{(x_1, \ldots, x_m)}\geq\norm{(x_1, \ldots, x_k)}>\frac{\norm{x_n}_n}{n}, \,\,\,\textrm{for\,\,every}\,\, k\leq m <n.
\]
Now ii) gives the conclusion.
\end{proof}
\begin{cor}\label{cor1.6}
Let $x=(x_1, \ldots, x_n, \ldots)$ be a non zero vector of $X$, then there exist $k\geq2$ and $(d_n)_{n\geq1}\subseteq(0,1]$ such that $d_n=\left(1-\frac{1}{n+1}\right)$, for $n\geq k$ and \[
\norm x=\sum_{n=1}^{\infty}d_n\norm{x_n}_n.
\]
\end{cor}
\begin{proof}
The proof is immediate from i) and ii) of Proposition \ref{prop1.5}.
\end{proof}
\begin{defn}\label{def1.7}
Let $x, y\in X$. We will say that the norms of the vectors $x$ and $y$ have the same representation if there exist $(d_n)_{n\geq1}$, $(s_n)_{n\geq1}\subseteq(0,1]$ with $\norm x=\sum_{n=1}^{\infty}d_n\norm{x_n}_n$ and $\norm y=\sum_{n=1}^{\infty}s_n\norm{y_n}_n$ such that $d_n=s_n$, for every $n\in\mathbb{N}$. We will write $x\sim y$, when the norms of the vectors $x$ and $y$ have the same representation.
\end{defn}
We now consider the following example. Let $k>1$, we choose $x_n \in X_n$, for every $n\leq k-1$, such that $\norm{(x_1, \ldots, x_{k-1})}=\frac{1}{k}$ and $x_k \in X_k$ with $\norm{x_k}_k=1$. Firstly, we observe that $\max\set{\frac{\norm{x_k}_k}{k}, \norm{(x_1, \ldots, x_{k-1})}}$ can be either of the terms $\frac{\norm{x_k}_k}{k}$ and $\norm{(x_1, \ldots, x_{k-1})}$, so for the first $n$ coordinates of $x$ the coefficients $d_n$ can be chosen in either of the ways described in i) of Proposition \ref{prop1.5}, hence the representation of the norm of the vector $x=(x_1, \ldots, x_k)$ is not unique. Further, we consider the vectors $y=(x_1, \ldots, x_{k-1}, y_k)$ and $z=(x_1, \ldots, x_{k-1}, z_k)$, where $y_k, z_k \in X_k$ such that $\frac{\norm{z_k}_k}{k}<\norm{(x_1, \ldots, x_{k-1})}<\frac{\norm{y_k}_k}{k}$. Then we have that $x\sim y$, $x\sim z$, but is not true that $y\sim z$.

\begin{remark}\label{rem1.8}
Let $x=(x_1, \ldots, x_n, \ldots)$, $y=(y_1, \ldots, y_n, \ldots)\in X$ and $k\geq2$. Let also
\[
\norm{(x_1, \ldots, x_n)}\geq \frac{\norm{x_{n+1}}_{n+1}}{n+1}\,\,\Rightarrow\,\, \norm{(y_1, \ldots, y_n)}\geq\frac{\norm{y_{n+1}}_{n+1}}{n+1}
\]
and
\[
\norm{(x_1, \ldots, x_n)}\leq \frac{\norm{x_{n+1}}_{n+1}}{n+1}\,\,\Rightarrow\,\, \norm{(y_1, \ldots, y_n)}\leq\frac{\norm{y_{n+1}}_{n+1}}{n+1}
\]
for every $1\leq n\leq k$.
\\
Then, whenever one of the following holds, $x\sim y$.
\begin{itemize}
\item[i)] $x_n=0=y_n$, for every $n>k$.
\item[ii)] $\norm x, \norm y\leq M$ and $\norm{(x_1, \ldots, x_k)}$, $\norm{(y_1, \ldots, y_k)}>\frac{2M}{k+1}$.
\end{itemize}
Indeed, i) is apparent by Definition \ref{def1.7} and ii) derives from Proposition \ref{prop1.1} and ii) of Proposition \ref{prop1.5}.
\end{remark}
\begin{defn}\label{def1.9}
Let $x_m=(x_{m1}, \ldots, x_{mn}, \ldots)$, $m\geq1$ be a sequence in $X$. Let also $x=(x_1, \ldots, x_n, \ldots)\in X$. The sequence $(x_m)$ is said to be pointwise convergent to $x$, if $x_{mn}\overset{\norm\cdot_n}\to x_n$, $m\to\infty$, for every $n\in\mathbb{N}$. To state the pointwise convergence of $(x_m)$ to $x$, we will write $x_m\overset{p}\rightarrow x$.
\end{defn}
\begin{prop}\label{prop1.10}
Let $(x_m)_{m\geq1}$ be a bounded sequence in $X$ and $x\in X\backslash\set{0}$ with $x_m\overset{p}\to x$. There exists an infinite subset $M$ of $\mathbb{N}$ such that the norms of the vectors of the set $\set{x_m : m \in M}\cup\set{x}$ have the same representation.
\end{prop}
\begin{proof}
Since $(x_m)_{m\geq1}$ is bounded, there exists $L>0$ such that $\sup(\{\norm{x_{mn}}_n:$ $ m,n \in\mathbb{N}\}$ $\cup\set{\norm{x_n}_n: n\in\mathbb{N}})\leq L$. As in iii) of Proposition \ref{prop1.5}, there exists $k_0\in\mathbb{N}$ such that
\[
\norm{(x_1, \ldots, x_{k_0})}>\frac{L}{n}, \,\,\, \textrm{for\,\, every}\,\, n>k_0.
\]
By the pointwise convergence of $(x_{m})$ to $x$, there exists $N\in\mathbb{N}$ such that
\[
\norm{(x_{m1}, \ldots, x_{mk_{0}})}>\frac{L}{n}, \,\,\, \textrm{for\,\, every}\,\, n>k_0\,\,\textrm{and}\,\, m\geq N.
\]
Consequently, we have that $\norm{(x_1, \ldots, x_{k_0})}>\frac{\norm{x_{n}}_n}{n}$ and $\norm{(x_{m1}, \ldots, x_{mk_0})}>\frac{\norm{x_{mn}}_n}{n}$, for every $n>k_0$ and $m\geq N$. Now ii) of Proposition \ref{prop1.5} yields that
\[
\norm x=\norm{(x_1, \ldots, x_{k_0})}+\sum_{n=k_0+1}^{\infty}\left(1-\frac{1}{n+1}\right)\norm{x_{n}}_n
\]
and
\[
\norm{x_m}=\norm{(x_{m1}, \ldots, x_{mk_0})}+\sum_{n=k_0+1}^{\infty}\left(1-\frac{1}{n+1}\right)\norm{x_{mn}}_n,
\]
$m\geq N$ (1).

Now it suffices to show that there exists an infinite subset $M$ of $\{m\in \mathbb{N}:$ $ m\geq N\}$ such that the norm of the vectors of the set $\set{x_m: m\in M}\cup\set{x}$ have the same representation. By Remark \ref{rem1.8}, it would be sufficient to show that that there exists an infinite subset $M$ of $\set{m\in \mathbb{N}: m\geq N}$ such that for any $n=1, \ldots, k_0-1$ we have
\[
\norm{(x_1, \ldots, x_{n})}\geq\frac{\norm{x_{n+1}}_{n+1}}{n+1}\,\,\Rightarrow\,\,
\norm{(x_{m1}, \ldots, x_{mn})}\geq\frac{\norm{x_{mn+1}}_{n+1}}{n+1}
\]
and
\[
\norm{(x_1, \ldots, x_{n})}\leq\frac{\norm{x_{n+1}}_{n+1}}{n+1}\,\,\Rightarrow\,\,
\norm{(x_{m1}, \ldots, x_{mn})}\leq\frac{\norm{x_{mn+1}}_{n+1}}{n+1},
\]
for every $m\in M$ (2).

For any $n=1, \ldots, k_0-1$ we have that either a) $\norm{(x_1, \ldots, x_n)}>\frac{\norm{x_{n+1}}_{n+1}}{n+1}$ or b) $\norm{(x_1, \ldots, x_n)}<\frac{\norm{x_{n+1}}_{n+1}}{n+1}$ or c) $\norm{(x_1, \ldots, x_n)}=\frac{\norm{x_{n+1}}_{n+1}}{n+1}$. Again, by the pointwise convergence of $(x_m)$ to $x$, the number $N$ in (1)  can be assumed large enough, so whenever one of a) or b) holds, the corresponding inequality holds for the vector $x_m$, for every $m\geq N$. Let now $\norm{(x_1, \ldots, x_n)}=\frac{\norm{x_{n+1}}_{n+1}}{n+1}$ for some $n\in\set{1, \ldots, k_0-1}$, then there exists an infinite subset $B$ of $\set{m\in \mathbb{N}: m\geq N}$ such that one of the following alternatives, $\norm{(x_{m1}, \ldots, x_{mn})}>\frac{\norm{x_{mn+1}}_{n+1}}{n+1}$ or $\norm{(x_{m1}, \ldots, x_{mn})}<\frac{\norm{x_{mn+1}}_{n+1}}{n+1}$ or $\norm{(x_{m1}, \ldots, x_{mn})}=\frac{\norm{x_{mn+1}}_{n+1}}{n+1}$, holds true for all $m\in B$. Applying the last argument inductively, we find an infinite subset $M$ of $\set{m\in\mathbb{N}: m\geq N}$ such that property (2) is satisfied, so we are done.
\end{proof}


\section{}

This section is devoted to the proof of our main result. From now on, the spaces $X_n$ are assumed to be of finite dimension (each having a normalized 1-unconditional basis).

Let $Y$ be a (real) Banach space with $\dim Y=n$ and  $e_1, \ldots, e_n$ be a basis of $Y$. A sequence $y_m=\sum_{k=1}^{n}a_{mk}e_k$, $m\geq1$ in $Y$ will be called monotone if it is monotone in every coordinate.

We will usually write  $x=\sum_{n=1}^{\infty}x_n$ for the element $x=(x_1, \ldots, x_n, \ldots)$ of $X$.
\begin{prop}\label{prop2.1}
The unit ball $B_{X}$ of $X$ is compact in the topology of pointwise convergence defined in Definition \ref{def1.9}
\end{prop}
\begin{proof}
Since each $X_n$ is finite dimensional, we have that $X_n$ is isometric to some $Y_{n}^{*}$. Set $E=(\sum_{n=1}^{\infty}\oplus Y_n)_0$, then clearly $E^{*}$ is isometric to $Z=(\sum_{n=1}^{\infty}\oplus X_n)_1$. The fact that the space $c_{00}((Y_n))$ is dense in $E$ implies that the weak* topology of the space $Z\cong X$ coincides on the bounded subsets of $X$ with the topology of pointwise convergence of Definition \ref{def1.9}. It follows that our assertion is immediate consequence of the following:
\\
\\
\textit{\underline{Claim.}
Let $(x_m)\subseteq X$ with $x_m=\sum_{n=1}^{\infty}x_{mn}$ and $\norm{x_{m}}\leq1$, for $m\in\mathbb{N}$. Let also $x\in X$ with $x=\sum_{n=1}^{\infty}x_n$ such that $x_m \overset{p}\rightarrow x$. Then $\norm x\leq 1$. (That is, $B_X$ is a weak* -- closed subset of $X$.)}
\\
\\
\textit{\underline{Proof of the Claim.}}
We assume that $\norm x>1$. Then, by Proposition \ref{prop1.4}, there exist $k\in\mathbb{N}$ and $\varepsilon>0$ such that
\begin{equation}\label{eq2.1}
\|\sum_{n=1}^{k}x_n\|>1+\varepsilon.
\end{equation}
Since $x_m\overset{p}\rightarrow x$, there also exists $m\in\mathbb{N}$ such that
\begin{equation}\label{eq2.2}
\norm{x_{mn}-x_n}_n<\frac{\varepsilon}{2k},\,\,\, n\leq k.
\end{equation}
Then
\begin{eqnarray*}
\norm{x_m} & = & \|\sum_{n=1}^{\infty}x_{mn}\| \\
& \geq & \|\sum_{n=1}^{k}x_{mn}\|, \,\, \textrm{by}\,\,\, \textrm{Proposition} \,\, \ref{prop1.4} \\
& = & \|\sum_{n=1}^{k}x_{mn}+\sum_{n=1}^{k}x_n-\sum_{n=1}^{k}x_n\| \\
& \geq & \|\sum_{n=1}^{k}x_{n}\|- \|\sum_{n=1}^{k}(x_{mn}-x_n)\| \\
& \geq & \|\sum_{n=1}^{k}x_n\|-\sum_{n=1}^{k}\norm{x_{mn}-x_n}_n\\
& > & 1+\varepsilon-\frac{\varepsilon}{2}=1+\frac{\varepsilon}{2}, \,\,\textrm{from}\,\,\, \eqref{eq2.1}\,\, \textrm{and}\,\, \eqref{eq2.2},
\end{eqnarray*}
a contradiction.
\end{proof}
\begin{thm}\label{thm2.2}
Let $(X_n)_{n\geq1}$ be a sequence of finite dimensional Banach spaces, each having an 1-unconditional basis. Then the Banach space $X$ does not contain an equilateral sequence.
\end{thm}
\begin{proof}
We are going to prove that the assumption, $X$ contains an equilateral sequence, leads to a contradiction. So let $x_m=\sum_{n=1}^{\infty}x_{mn}$, $m\geq1$ be an equilateral sequence in $X$. First, we observe that we may assume that $\norm{x_m}=1$ and $\norm{x_m-x_k}=1$, for $m\neq k\in\mathbb{N}$. By Proposition \ref{prop2.1}, it may be assumed further that there exists $x\in X$ such that $x_m \overset{p}\rightarrow x$ and $\norm x\leq1$.

The proof that our assumption leads to a contradiction is divided in 5 steps.
\\
\\
\textit{\underline{Step 1.}
By passing to a subsequence, we may have that $\norm{x_m-x}=\frac{1}{2}$, for every $m\in\mathbb{N}$.}
\\
\\
\textit{\underline{Proof of Step 1.}}
We first observe that there exists at most one $m\in\mathbb{N}$ such that $\norm{x_m-x}<\frac{1}{2}$. Indeed, assuming that there exist $s\neq l\in \mathbb{N}$ such that $\norm{x_i-x}<\frac{1}{2}$, for $i=s, l$, we would have that
\[
1=\norm{x_s-x_l}\leq\norm{x_s-x}+\norm{x_l-x}<\frac{1}{2}+\frac{1}{2}=1,
\]
a contradiction. Therefore, we may assume that $\norm{x_m-x}\geq\frac{1}{2}$, for every $m\in\mathbb{N}$.

We are going to prove that $\norm{x_m-x}>\frac{1}{2}$, for finitely many $m\in\mathbb{N}$. We suppose that there exists a subsequence of $(x_m)$, denoted again by $(x_m)$, such that $\norm{x_m-x}>\frac{1}{2}$, for every $m\in\mathbb{N}$. Then for every $m\in\mathbb{N}$ there exists $b_m>0$ such that
\begin{equation}\label{eq2.3}
\norm{x_m-x}=\frac{1}{2}+b_m.
\end{equation}
Since $\norm x\leq1$ and $\norm{x_m}=1$, for every $m\in\mathbb{N}$, by \eqref{eq2.3} we take that
\begin{equation}\label{eq2.4}
0<b_m\leq\frac{3}{2}, \,\,\,\textrm{for\,\, every} \,\,m\in\mathbb{N}.
\end{equation}
We fix $m\in\mathbb{N}$. Then there exist $k\in\mathbb{N}$ and $l\geq k$ such that
\begin{equation}\label{eq2.5}
\|\sum_{n=1}^{k}(x_{mn}-x_n)\|>\frac{1}{2}+\frac{b_m}{2}
\end{equation}
and
\begin{equation}\label{eq2.6}
\|\sum_{n=l+1}^{\infty}(x_{mn}-x_n)\|<\frac{b_m}{8}.
\end{equation}
Indeed, both \eqref{eq2.5} and \eqref{eq2.6} are consequences of the fact that $\sum_{n=1}^{k}(x_{mn}-x_n)\overset{\norm\cdot}\rightarrow x_m-x$, $k\rightarrow \infty$. Further, $l$ can be chosen quite large such that
\begin{equation}\label{eq2.7}
\frac{4}{n}<\frac{1}{2}+\frac{b_m}{2}
\end{equation}
and
\begin{equation}\label{eq2.8}
\frac{2}{n+2}<\frac{b_m}{8},\,\, \textrm{for\,\, every}\,\,n\geq l.
\end{equation}
We note that \eqref{eq2.7} and \eqref{eq2.4} yield that
\begin{equation}\label{eq2.9}
\frac{2(\frac{1}{2}+b_j)}{n}<\frac{1}{2}+\frac{b_m}{2}, \,\,\textrm{for\,\, every} \,\, n\geq l \,\,\textrm{and} \,\, j\in\mathbb{N}.
\end{equation}
Since $x_m\overset{p}\rightarrow x$, we can choose $s\in\mathbb{N}$ such that
\begin{equation}\label{eq2.10}
\|\sum_{n=1}^{l}(x_{sn}-x_n)\|<\frac{b_m}{8}.
\end{equation}
The last, by Proposition \ref{prop1.4}, yields that
\[
\|\sum_{n=1}^{k}(x_{sn}-x_n)\|<\frac{b_m}{8},\,\, \textrm{since}\,\,l\geq k.
\]
We have that $\|\sum_{n=1}^{\infty}(x_{sn}-x_n)\|=\norm{x_s-x}=\frac{1}{2}+b_s$. So, by Proposition \ref{prop1.3}, $\norm{x_{sn}-x_n}_{n}\leq2(\frac{1}{2}+b_s)$, for every $n\in\mathbb{N}$. The last, combined with \eqref{eq2.9} and \eqref{eq2.5}, gives that
\[
\frac{\norm{x_{sn}-x_n}_n}{n}\leq\frac{1}{2}+\frac{b_m}{2}<
\|\sum_{n=1}^{k}(x_{mn}-x_n)\|,\,\,\,\textrm{for\,\,every}\,\, n\geq l.
\]
Now, by ii) of Proposition \ref{prop1.5}, we take that
\begin{equation}\label{eq2.11}
\begin{split}
\|\sum_{n=1}^{k}(x_{mn}-x_n)+\sum_{n=l+1}^{\infty}(x_n-x_{sn})\|=
\|\sum_{n=1}^{k}(x_{mn}-x_n)\|+ \\
& \!\!\!\!\!\!\!\!\!\!\!\!\!\!\!\!\!\!\!\!\!\!\!\!\!\!\!\!\!
+ \sum_{n=l+1}^{\infty}\left(1-\frac{1}{n+1}\right)\norm{x_n-x_{sn}}_n.
\end{split}
\end{equation}
Finally, we have that,
\begin{eqnarray*}
\norm{x_m-x_s} & = & \|\sum_{n=1}^{\infty}(x_{mn}-x_{sn})\| \\
& \geq & \|\sum_{n=1}^{k}(x_{mn}-x_{sn})+\sum_{n=l+1}^{\infty}(x_{mn}-x_{sn})\|,
\,\,\,\textrm{by Proposition}\,\, \ref{prop1.4}
\end{eqnarray*}
\begin{eqnarray*}
& = & \|\sum_{n=1}^{k}(x_{mn}-x_{sn})+\sum_{n=l+1}^{\infty}(x_{mn}-x_{sn})\pm
\sum_{n=1}^{k}(x_n-x_{sn})\pm\sum_{n=l+1}^{\infty}(x_{mn}-x_n)\|\\
& \geq & \|\sum_{n=1}^{k}(x_{mn}-x_n)+\sum_{n=l+1}^{\infty}(x_n-x_{sn})\|-
\|\sum_{n=1}^{k}(x_n-x_{sn})\|-\|\sum_{n=l+1}^{\infty}(x_{mn}-x_n)\|\\
& \geq & \|\sum_{n=1}^{k}(x_{mn}-x_n)\|+\sum_{n=l+1}^{\infty}\left(1-\frac{1}{n+1}\right)\|x_n-x_{sn}\|_n
-\frac{b_m}{4}, \,\,\textrm{by}\,\,\eqref{eq2.6},\,\eqref{eq2.10}\,\,\textrm{and}\,\, \eqref{eq2.11} \\
& \geq & \|\sum_{n=1}^{k}(x_{mn}-x_n)\|+\left(1-\frac{1}{l+2}\right)\sum_{n=l+1}^{\infty}\|x_n-x_{sn}\|_{n}
-\frac{b_m}{4} \\
& \geq & \|\sum_{n=1}^{k}(x_{mn}-x_n)\|+\left(1-\frac{1}{l+2}\right)\|\sum_{n=l+1}^{\infty}(x_n-x_{sn})\|
-\frac{b_m}{4},
\end{eqnarray*}
by the right hand side inequality of Proposition \ref{prop1.1}
\begin{eqnarray*}
& = & \|\sum_{n=1}^{k}(x_{mn}-x_n)\|+\left(1-\frac{1}{l+2}\right) \|\pm\sum_{n=1}^{l}(x_n-x_{sn})+\sum_{n=l+1}^{\infty}(x_n-x_{sn})\|-\frac{b_m}{4}\\
& = & \|\sum_{n=1}^{k}(x_{mn}-x_n)\|+\left(1-\frac{1}{l+2}\right)\|\sum_{n=1}^{\infty}(x_n-x_{sn})
-\sum_{n=1}^{l}(x_n-x_{sn})\|-\frac{b_m}{4}\\
& \geq & \|\sum_{n=1}^{k}(x_{mn}-x_n)\|+\left(1-\frac{1}{l+2}\right)\|\sum_{n=1}^{\infty}(x_n-x_{sn})\|-
\|\sum_{n=1}^{l}(x_n-x_{sn})\|-\frac{b_m}{4}\\
& > & \frac{1}{2}+\frac{b_m}{2}+\left(1-\frac{1}{l+2}\right)\left(\frac{1}{2}+b_s\right)-
\left(1-\frac{1}{l+2}\right)\frac{b_m}{8} - \frac{b_m}{4}, \,\,\textrm{by}\,\, \eqref{eq2.5}\,\, \textrm{and}\,\, \eqref{eq2.10}\\
& = & \frac{1}{2}+\frac{b_m}{4}+\left(1-\frac{1}{l+2}\right)\left(\frac{1}{2}+b_s\right)-
\left(1-\frac{1}{l+2}\right)\frac{b_m}{8} \\
& \geq & \frac{1}{2}+\frac{b_m}{4}+\frac{1}{2}+b_s-\frac{1}{l+2}\left(\frac{1}{2}+b_s\right)-
\left(1-\frac{1}{l+2}\right)\frac{b_m}{8} \\
& \geq & 1+\frac{b_m}{4}+b_s-\frac{1}{l+2}\left(\frac{1}{2}+\frac{3}{2}\right)-
\left(1-\frac{1}{l+2}\right)\frac{b_m}{8},\,\, \textrm{since\,\,by}\,\,\eqref{eq2.4}\,\,b_s\leq\frac{3}{2} \\
& \geq & 1+\frac{b_m}{4}+b_s-2\frac{b_m}{8}, \,\,\textrm{by}\,\,\eqref{eq2.8}\\
& = & 1+b_s>1,
\end{eqnarray*}
a contradiction.
\\
\\
\textit{\underline{Step 2.}}
$\norm x=\frac{1}{2}$.
\\
\\
\underline{\textit{Proof of Step 2.}}
For any $m\in\mathbb{N}$,
\[
\frac{1}{2}=\norm{x_m-x}\geq\norm{x_m}-\norm x=1-\norm x
\]
hence, $\norm x\geq\frac{1}{2}$. We assume that $\norm x>\frac{1}{2}$. By our assumption there exist $b>0$ and $k\in\mathbb{N}$ such that
\begin{equation}\label{eq2.12}
\|\sum_{n=1}^{k}x_n\|>\frac{1}{2}+b
\end{equation}
and
\begin{equation}\label{eq2.13}
\|\sum_{n=k+1}^{\infty}x_n\|<\frac{b}{4}.
\end{equation}
In addition, $k$ can be chosen large enough such that
\begin{equation}\label{eq2.14}
\frac{1}{2(k+2)}<\frac{b}{4}
\end{equation}
and
\begin{equation}\label{eq2.15}
\frac{\|x_{mn}-x_n\|_{n}}{n}<\|\sum_{n=1}^{k}x_n\|, \,\, \textrm{for\,\, every}\,\,m\in\mathbb{N}\,\,\textrm{and}\,\, n\geq k.
\end{equation}
\eqref{eq2.12} and \eqref{eq2.13} come directly by the fact $\sum_{n=1}^{s}x_n \overset{\norm\cdot}\to x$, $s\to \infty$. For \eqref{eq2.15}, by Proposition \ref{prop1.1},
\[
\sum_{n=1}^{\infty}\frac{1}{2}\|x_{mn}-x_n\|_n\leq\|x_m-x\|=1,\,\, \textrm{for\,\, every}\,\,m\in\mathbb{N},
\]
so $\sum_{n=1}^{\infty}\norm{x_{mn}-x_n}_{n}\leq2$ and consequently, $\norm{x_{mn}-x_n}_{n}\leq 2$, for every $m, n\in \mathbb{N}$. Then there exists $s\in\mathbb{N}$ such that $\frac{\|x_{mn}-x_n\|_n}{n}<\frac{1}{2}+b$, for every $m\in\mathbb{N}$ and $n\geq s$.

By the pointwise convergence of $(x_m)$ to $x$, we choose $m\in\mathbb{N}$ such that
\begin{equation}\label{eq2.16}
\|\sum_{n=1}^{k}(x_{mn}-x_n)\|<\frac{b}{4}.
\end{equation}
Now we have that
\begin{eqnarray*}
\norm{x_{m}} & = & \|\sum_{n=1}^{\infty}x_{mn}\|  \\
& = & \|\sum_{n=1}^{\infty}x_{mn}\pm\sum_{n=1}^{\infty}x_n\| \\
& = & \|\sum_{n=1}^{k}(x_{mn}-x_n)+\sum_{n=k+1}^{\infty}(x_{mn}-x_{n})+
\sum_{n=1}^{k}x_n+\sum_{n=k+1}^{\infty}x_{n}\| \\
& \geq & \|\sum_{n=1}^{k}x_n +\sum_{n=k+1}^{\infty}(x_{mn}-x_n)\|-\|\sum_{n=1}^{k}(x_{mn}-x_n)\|-\|\sum_{n=k+1}^{\infty}x_n\| \\
& \geq & \|\sum_{n=1}^{k}x_n+\sum_{n=k+1}^{\infty}(x_{mn}-x_n)\|-\frac{b}{2},
\,\,\textrm{by}\,\, \eqref{eq2.13}\,\, \textrm{and}\,\, \eqref{eq2.16}\\
& \geq & \|\sum_{n=1}^{k}x_n\|+\left(1-\frac{1}{k+2}\right)\sum_{n=k+1}^{\infty}\|x_{mn}-x_n\|_n-\frac{b}{2},
\,\,\textrm{by}\,\, \eqref{eq2.15}\\
& & \,\,\,\textrm{and\,\, ii)\,\, of \,\,Proposition} \,\, \ref{prop1.5} \\
& \geq & \|\sum_{n=1}^{k}x_n\|+\left(1-\frac{1}{k+2}\right)\|\sum_{n=k+1}^{\infty}(x_{mn}-x_n)\|-\frac{b}{2},\,\,\,
\textrm{by\,\,Proposition} \,\,\ref{prop1.1}\\
& = & \|\sum_{n=1}^{k}x_n\|+\left(1-\frac{1}{k+2}\right)\|x_m-x-\sum_{n=1}^{k}(x_{mn}-x_n)\|-\frac{b}{2} \\
& \geq & \|\sum_{n=1}^{k}x_n\|+\left(1-\frac{1}{k+2}\right)
\left(\|x_m-x\|-\|\sum_{n=1}^{k}(x_{mn}-x_n)\|\right)-\frac{b}{2} \\
& > & \frac{1}{2}+b+\left(1-\frac{1}{k+2}\right)\left(\frac{1}{2}-\frac{b}{4}\right)-
\frac{b}{2},\,\, \textrm{by}\,\,\eqref{eq2.12}\,\,\textrm{and}\,\,\eqref{eq2.16}  \\
& = & \frac{1}{2}+b+\frac{1}{2}-\frac{1}{2(k+2)}-\frac{b}{4}+\frac{b}{4(k+2)}-\frac{b}{2}\\
& > & 1+\frac{b}{4(k+2)}, \,\,\textrm{by}\,\, \eqref{eq2.14},
\end{eqnarray*}
a contradiction.
\\
\\
\textit{\underline{Step 3.}
The following assertion is not true. For every $l\in\mathbb{N}$ there exists $m_0\in\mathbb{N}$ such that $x_{mn}=x_n$, for every $n\leq l$ and $m\geq m_0$.}
\\
\\
\textit{\underline{Proof of Step 3.}}
We assume that the assertion is true. Then there exists a subsequence of $(x_m)$, denoted again by $(x_m)$, and a strictly increasing sequence $(t(m))$ in $\mathbb{N}$ such that
\begin{equation}\label{eq2.17}
x_m=\sum_{n=1}^{t(m)}x_n+\sum_{n=t(m)+1}^{\infty}x_{mn},\,\, \textrm{for\,\, every}\,\, m\in\mathbb{N}.
\end{equation}
By Remark \ref{rem1.2} and \eqref{eq2.17}, we take that
\begin{equation}\label{eq2.18}
\norm{x_m-x}=\|\sum_{n=t(m)+1}^{\infty}(x_{mn}-x_n)\|>
\sum_{n=t(m)+1}^{\infty}\left(1-\frac{1}{n+1}\right)\|x_{mn}-x_n\|_n,
\end{equation}
for every $m\in\mathbb{N}$. We now fix $m\in\mathbb{N}$. Then, by iii) of Proposition \ref{prop1.5}, there exists $k_0>m$ such that
\begin{equation}\label{eq2.19}
\|x_m-x\|=\|\sum_{n=t(m)+1}^{t(k)}(x_{mn}-x_n)\|+
\sum_{n=t(k)+1}^{\infty}\left(1-\frac{1}{n+1}\right)\norm{x_{mn}-x_n}_n,
\end{equation}
for every  $k\geq k_0$.
Since by \eqref{eq2.18} $\|\sum_{n=t(m)+1}^{s}(x_{mn}-x_n)\| \to \|x_m-x\|=\frac{1}{2}$, we choose $k\geq k_0$ such that
\begin{equation}\label{eq2.20}
\|\sum_{n=t(m)+1}^{t(k)}(x_{mn}-x_{kn})\|=\|\sum_{n=t(m)+1}^{t(k)}(x_{mn}-x_n)\|>\frac{1}{4}
\end{equation}
and
\begin{equation}\label{eq2.21}
\frac{2}{t(k)}<\frac{1}{4}.
\end{equation}
Since by Proposition \ref{prop1.1}, $\frac{1}{2}\sum_{n=1}^{\infty}\|x_{mn}-x_{kn}\|_{n}\leq\|x_m-x_k\|=1$, we take that $\norm{x_{mn}-x_{kn}}_n\leq2$, for every $n\in\mathbb{N}$. Therefore
\begin{equation}\label{eq2.22}
\frac{\norm{x_{mn}-x_{kn}}_n}{n}\leq\frac{2}{t(k)}<\frac{1}{4},\,\, \textrm{for\,\, every}\,\, n\geq t(k).
\end{equation}
Properties \eqref{eq2.20}, \eqref{eq2.21}, \eqref{eq2.22} and ii) of Proposition \ref{prop1.5} now give that

\begin{eqnarray*}
\norm{x_m-x_k} & = & \|\sum_{n=t(m)+1}^{t(k)}(x_{mn}-x_{kn})\| + \sum_{n=t(k)+1}^{\infty}\left(1-\frac{1}{n+1}\right)\norm{x_{mn}-x_{kn}}_n \\
& = & \|\sum_{n=t(m)+1}^{t(k)}(x_{mn}-x_{n})\| + \sum_{n=t(k)+1}^{\infty}\left(1-\frac{1}{n+1}\right)\norm{x_{mn}-x_{kn}}_n,\\
& & \textrm{by}\,\, \eqref{eq2.20} \\
& = & \|\sum_{n=t(m)+1}^{t(k)}(x_{mn}-x_{n})\| + \sum_{n=t(k)+1}^{\infty}\left(1-\frac{1}{n+1}\right)\norm{x_{mn}-x_{kn}}_n\pm \\
& & \qquad\qquad\qquad\qquad\qquad
\pm\sum_{n=t(k)+1}^{\infty}\left(1-\frac{1}{n+1}\right)\norm{x_{mn}-x_n}_n \\
& \leq & \|\sum_{n=t(m)+1}^{t(k)}(x_{mn}-x_{n})\| + \sum_{n=t(k)+1}^{\infty}\left(1-\frac{1}{n+1}\right)\norm{(x_{mn}-x_{kn})-(x_{mn}-x_n)}_n\\
& & \qquad\qquad\qquad\qquad\qquad
+\sum_{n=t(k)+1}^{\infty}\left(1-\frac{1}{n+1}\right)\norm{x_{mn}-x_n}_n \\
& = & \|\sum_{n=t(m)+1}^{t(k)}(x_{mn}-x_{n})\| + \sum_{n=t(k)+1}^{\infty}\left(1-\frac{1}{n+1}\right)\norm{x_{n}-x_{kn}}_n \\
& & \qquad\qquad\qquad\qquad\qquad
+\sum_{n=t(k)+1}^{\infty}\left(1-\frac{1}{n+1}\right)\norm{x_{mn}-x_n}_n \\
& = & \norm{x_m-x}+\sum_{n=t(k)+1}^{\infty}\left(1-\frac{1}{n+1}\right)\norm{x_n-x_{kn}}_n ,\,\, \textrm{by}\,\, \eqref{eq2.19}  \\
& < & \norm{x_m-x}+\norm{x_k-x},\,\, \textrm{by}\,\, \eqref{eq2.18} \\
& = & \frac{1}{2}+\frac{1}{2}=1,
\end{eqnarray*}
a contradiction.

Step 3 informs us that there exists $l\in\mathbb{N}$ such that for every $m\in \mathbb{N}$ there exist $n\leq l$ and $k\geq m$ such that $x_{kn}\neq x_n$. Since the set $\set{1, \ldots, l}$ is finite there exist $n_0 \leq l$ and a subsequence of $(x_m)$, denoted again by $(x_m)$, such that $x_{mn_{0}}\neq x_{n_{0}}$, for every $m\in\mathbb{N}$.
By iii) of Proposition \ref{prop1.5}, there exists a strictly increasing sequence $(t(m))$ in $\mathbb{N}$ such that $t(1)\geq n_{0}$ and
\[
\norm{x_m-x}=\|\sum_{n=1}^{t(m)}(x_{mn}-x_{n})\|+
\sum_{n=t(m)+1}^{\infty}\left(1-\frac{1}{n+1}\right)\norm{x_{mn}-x_n}_n,
\,\,\textrm{for\,\,every}\,\, m\in\mathbb{N}.
\]
Corollary \ref{cor1.6} now gives that for every $m\in\mathbb{N}$ there exists $(A_{mn})_{n=1}^{t(m)}\subseteq[0,1]$ such that
\[
\norm{x_m-x}=\sum_{n=1}^{t(m)}A_{mn}\norm{x_{mn}-x_n}_{n}+
\sum_{n=t(m)+1}^{\infty}\left(1-\frac{1}{n+1}\right)\norm{x_{mn}-x_n}_n.
\]
\newline
\textit{\underline{Step 4.}
Let $m\in\mathbb{N}$, then there exists an infinite subset $M$ of $\mathbb{N}$ such that
\[
\|\sum_{n=1}^{n_0}(x_{mn}-x_{kn})+\sum_{n=n_0+1}^{\infty}(x_{mn}-x_n)\| =
\]
\[
=
\sum_{n=1}^{n_0}A_{mn}\norm{x_{mn}-x_{kn}}_n+\sum_{n=n_{0}+1}^{t(m)}A_{mn}\norm{x_{mn}-x_n}_n+
\sum_{n=t(m)+1}^{\infty}\left(1-\frac{1}{n+1}\right)\norm{x_{mn}-x_n}_n,
\]
for every $k\in M$.
}
\\
\\
\textit{\underline{Proof of Step 4.}}
We observe that $\sum_{n=1}^{n_0}(x_{mn}-x_{kn})+\sum_{n=n_0+1}^{\infty}(x_{mn}-x_n)\overset{p}\to x_m-x$, $k\to\infty$. Now the conclusion is direct from Proposition \ref{prop1.10}.
\\
\\
For the remaining of this proof we are going to assume that $(x_{mn})_{m\geq1}\subseteq X_{n}$ is monotone for every $1\leq n\leq n_0$. This is proved by a standard diagonal argument.
\\
\\
\textit{\underline{Step 5.}
For $m\in\mathbb{N}$ we denote by $M_{m}$ an infinite subset of $\mathbb{N}$ as in Step 4. There exist $m\in\mathbb{N}$ and $k\in M_{m}$ such that}
\[
\sum_{n=1}^{n_{0}}A_{mn}\norm{x_{mn}-x_{kn}}_n<\sum_{n=1}^{n_0}A_{mn}\norm{x_{mn}-x_n}_n.
\]
\textit{\underline{Proof of Step 5.}}
Firstly, it is obvious that we can assume that $m<\min M_{m}$, for $m\in\mathbb{N}$. Also we may assume that $M_{k'}\subseteq M_k$, for $k<k'\in\mathbb{N}$. Let now $m\in\mathbb{N}$, $d_n=\dim X_n$ and $\set{e_{1}^{n}, \ldots, e_{d_n}^{n}}$ be an 1-unconditional basis of $X_n$, for $1\leq n\leq n_0$. Let also
\[
x_{sn}=\sum_{j=1}^{d_n}a_{sn_j}e_{j}^{n}
\,\,\,\,
\textrm{and}
\,\,\,\,
x_n=\sum_{j=1}^{d_n}a_{n_j}e_{j}^{n}
\]
be the expressions of the vectors $x_{sn}\in X_n$ and $x_n \in X_n$, respectively, on the basis $\set{e_{1}^{n}, \ldots, e_{d_n}^{n}}$, for $s\in\mathbb{N}$ and $1\leq n\leq n_0$. By the monotonicity of the sequences $(x_{mn})_{m\geq1}$,
for $1\leq n\leq n_0$, we have that $|a_{mn_{j}}-a_{kn_{j}}|\leq|a_{mn_{j}}-a_{n_j}|$ for any $k\in M_{m}$, $1\leq n\leq n_0$ and $1\leq j\leq d_n$. Thus for $k\in M_{m}$ there exist $(\lambda_{n_{j}}^{k})_{j=1}^{d_n}\subseteq [-1,1]$ such that
\[
x_{mn}-x_{kn}=\sum_{j=1}^{d_n}(a_{mn_{j}}-a_{kn_{j}})e_{j}^{n}=
\sum_{j=1}^{d_n}\lambda_{n_j}^{k}(a_{mn_{j}}-a_{n_j})e_{j}^{n}
\]
for $1\leq n\leq n_0$, with $\lambda_{n_{j}}^{k}=0$ when $a_{mn_{j}}=a_{n_j}$.

In fact $(\lambda_{n_{j}}^{k})_{j=1}^{d_n}\subseteq (-1,1]$ for any $k\in M_{m}$. Indeed, assuming that $\lambda_{n_{j}}^{k}=-1$, for some $1\leq n\leq n_0$ and $1\leq j\leq d_n$, then $a_{mn_{j}}\neq a_{n_j}$ and $2a_{mn_{j}}=a_{n_j}+a_{kn_{j}}$, a contradiction by the monotonicity of the sequences $(x_{mn})_{m\geq1}$, for $1\leq n\leq n_0$. Now we are going to use the following result: for any unconditional basis $(e_i)$ in a Banach space $X$, for all $m\in \mathbb{N}$, all scalars $(a_i)_{i=1}^{l}$ and all sequences $(\lambda_i)_{i=1}^{l}$, we have
\[
\|\sum_{i=1}^{l}\lambda_ia_ie_i\|\leq ubc\set{e_i}\max_{1\leq i\leq l}{|\lambda_i|} \cdot \|\sum_{i=1}^{l}a_ie_i\|,
\]
where $ubc\set{e_i}$ is the unconditional basis constant of the basis $(e_i)$. A proof of the last can be found in \cite{1}. By the 1-unconditionality of $\set{e_{1}^{n}, \ldots, e_{d_n}^{n}}$ and the above inequality, we have
\begin{eqnarray*}
\norm{x_{mn}-x_{kn}}_n & = & \|\sum_{j=1}^{d_n}\lambda_{n_j}^{k}(a_{mn_{j}}-a_{n_j})e_{j}^{n}\| \\
& \leq & ubc\set{e_{j}^{n}}\max_{1\leq j\leq d_n}|\lambda_{n_{j}}^{k}|\cdot \|\sum_{j=1}^{d_n}(a_{mn_{j}}-a_{n_j})e_{j}^{n}\|_n \\
& \leq & \max_{1\leq j\leq d_n}|\lambda_{n_{j}}^{k}|\cdot
\|\sum_{j=1}^{d_n}(a_{mn_{j}}-a_{n_j})e_{j}^{n}\|_n \\
& = & \max_{1\leq j\leq d_n}|\lambda_{n_{j}}^{k}|\cdot\|x_{mn}-x_n\|_n,
\end{eqnarray*}
for every $k\in M_m$ and $1\leq n\leq n_0$ \eqref{eq2.17}.

Let us assume that for every $m\in\mathbb{N}$ and $k\in M_m$, $\max_{1\leq j\leq d_{n_{0}}}|\lambda_{n_{0}j}^{k}|=1$. We put $B_s=\set{j\leq d_{n_{0}}: a_{sn_{0}j}=a_{n_{0}j}}$, for $s\in\mathbb{N}$. Now for $m\in\mathbb{N}$ and $k\in M_m$ the choice of $\lambda_{n_{0}j}^{k}$ and our last assumption yield that $B_m\varsubsetneq B_{k}$. So for $m=1$ and $k\in M_1$, we have that $B_1 \varsubsetneq B_{k}$. For $k'\in M_k$ we will also have that $B_{k}\varsubsetneq B_{k'}$. Inductively, since the set $\set{1, \ldots, d_{n_{0}}}$ is finite, we can find $s\in M_1$ such that $B_s=\set{1, \ldots, d_{n_{0}}}$, a contradiction, since $x_{sn_0}\neq x_{n_0}$, for every $s\in \mathbb{N}$.

The last along with \eqref{eq2.17} entails that there exist $m\in\mathbb{N}$ and $k\in M_{m}$ such that
\[
\norm{x_{mn_{0}}-x_{kn_{0}}}_{n_{0}}<\norm{x_{mn_{0}}-x_{n_{0}}}_{n_{0}},
\]
then by \eqref{eq2.17}, we get that
\[
\sum_{n=1}^{n_0}A_{mn}\norm{x_{mn}-x_{kn}}_n<\sum_{n=1}^{n_0}A_{mn}\norm{x_{mn}-x_n}_n.
\]
The proof of Step 5 is complete.
\\
\\
Finally, we are in position to prove that our initial assumption leads to a contradiction. Let $m\in\mathbb{N}$ and $k\in M_m$, as in Step 5. We put $y=\sum_{n=1}^{n_0}x_{kn}+\sum_{n=n_0+1}^{\infty}x_n$, then by Step 4
\begin{eqnarray*}
\norm{x_m-y} & = & \sum_{n=1}^{n_0}A_{mn}\norm{x_{mn}-x_{kn}}_n+\sum_{n=n_0+1}^{t(m)}A_{mn}\norm{x_{mn}-x_n}_n\\
& & \qquad\qquad\qquad\qquad\qquad
+ \sum_{t(m)+1}^{\infty}\left(1-\frac{1}{n+1}\right)\norm{x_{mn}-x_n}_n \\
& < & \sum_{n=1}^{n_0}A_{mn}\norm{x_{mn}-x_{n}}_n+\sum_{n=n_0+1}^{t(m)}A_{mn}\norm{x_{mn}-x_n}_n+\\
& & \qquad\qquad\qquad
+ \sum_{t(m)+1}^{\infty}\left(1-\frac{1}{n+1}\right)\norm{x_{mn}-x_n}_n, \,\,\textrm{by\,\, Step}\,\, 5\\
& =& \norm{x_m-x}. \,\,(*)
\end{eqnarray*}
Also,
\[
\norm{x_k-y}=\|\sum_{n=n_0+1}^{\infty}(x_{kn}-x_{n})\|\leq \|\sum_{n=1}^{\infty}(x_{kn}-x_n)\|=\|x_k-x\|.
\]
The last combined with (*) gives that
\begin{eqnarray*}
\norm{x_m-x_k} & \leq & \norm{x_{m}-y}+\norm{x_{k}-y} \\
& < & \norm{x_{m}-x}+\norm{x-x_k} \\
& = & \frac{1}{2}+\frac{1}{2} =1,
\end{eqnarray*}
a contradiction.
\end{proof}

\begin{cor}\textrm{(Terenzi \cite{4})}
There exists an equivalent norm $\norm\cdot$ on $l_1$ such that $(l_1, \norm\cdot)$ does not contain an equilateral sequence.
\end{cor}
\begin{proof}
We take $X_n=\Real$, $n\geq1$, in Theorem \ref{thm2.2}, then clearly $X\cong l_1$ and Terenzi's norm on $X$ satisfies the assertion.
\end{proof}
As it was mentioned in the Introduction the most part of our proof is essentially the same with Terenzi’s. The distinction between the two proofs lays on Step 5 of Theorem \ref{thm2.2}. Terenzi considers 1-dimensional Banach spaces, thus to prove Step 5 it suffices to observe that $x_{sn_0}\neq x_{n_0}$, for every $s\in \mathbb{N}$ and since the sequences $(x_{mn})_{m\geq1}$, for $1\leq n\leq n_0$ are monotone, we have that for $1\leq n\leq n_0$ either $x_{mn}\leq x_{kn}\leq x_n$ or $x_{mn}\geq x_{kn}\geq x_n$, where the inequalities are strict when $n=n_0$.

Terenzi constructed still another renorming of $l_1$ not admitting any equilateral sequence \cite{5}. Using Terenzi’s methods it can be proved that the space $l_1$ equipped with its canonical norm contains infinite dimensional subspaces which do not contain an equilateral sequence. By \cite{3} we can choose a weak* closed stictly convex subspace $X$ of $(l_1, || \cdot ||_1 )$ and we assume that there exists a normalized 1- equilateral sequence $(x_n)$ in $X$. The proof of Step 1, which here can be simplified, yields that there exists $x \in X$ with $||x|| \le 1$ such that $||x_n-x||=\frac{1}{2}$ , for $n \in \Bbb{N}$. Then for $n \in \Bbb{N}$
\[||(x_1-x) -(x_n-x)||=1=||x_1-x||+||x_n-x||. \]
By the strict convexity of $X$ this implies that $x_1-x=-(x_n-x)$ or $x_n=2x-x_1$, for every $n \in \Bbb{N}$, a contradiction. The aforementioned result and its proof were suggested to us from Professor P. Dowling. We include it here with his kind permission.

Theorem \ref{thm2.2} above provides us with further examples (beyond the space $l_1$) of Banach spaces, which do not contain any equilateral sequence.

\begin{cor}
There exists a separable Banach space with an unconditional basis not isomorphic to a subspace of $l_1$ and not containing an equilateral sequence.
\end{cor}
\begin{proof}
For $n\geq 1$, we take $X_n=l_p^{n}$, with $p>2$, in Theorem \ref{thm2.2}. Then the Banach space $(X, \norm\cdot)$ does not contain an equilateral sequence. On the other hand, $X$ can not be embedded in $l_1$, because its isomorphic version $Z=(\sum_{n=1}^{\infty}\oplus l_{p}^{n})_1$ has cotype$\geq p>2$ and the space $l_1$ has cotype 2 (see Th. 23, p.98 in \cite{6} and Remark 6.2.11, Th. 6.2.14 in \cite{1}).
\end{proof}

\vspace{0.5cm}
UNIVERSITY OF ATHENS, DEPARTMENT OF MATHEMATICS, \\PANEPISTIMIOUPOLIS, 15784 ATHENS, GREECE\\
\textit{E-mail address: e.glakousakis@gmail.com}\\ \\ \\
UNIVERSITY OF ATHENS, DEPARTMENT OF MATHEMATICS, \\PANEPISTIMIOUPOLIS, 15784 ATHENS, GREECE\\
\textit{E-mail address: smercour@math.uoa.gr}

\begin{thebibliography}{00}
\bibitem{1} F. Albiac and N. J. Kalton, \textit{Topics in Banach Space Theory}, Graduate Texts in Mathematics 233, Springer 2006.
\bibitem{2} D. Freeman, E. Odell, B. Sari and Th. Schlumprecht, Equilateral sets in uniformly smooth Banach Spaces, \textit{Mathematika (p.1 of 13) Univ. Coll. London} doi: 10.1112/S0025579313000260.
\bibitem{2A} M. I. Kadets and V. P. Fonf, Strictly convex subspaces of the space $l_1$, \textit{Mat. Zametki}, \textbf{33}(3) (1983), 417-422.
\bibitem{3} S. K. Mercourakis and G. Vassiliadis, Equilateral sets in infinite dimensional Banach Spaces, \textit{Proc. Amer. Math. Soc.} \textbf{142} (2014), 205-212.
\bibitem{4} P. Terenzi, Successioni regolari negli spazi di Banach, \textit{Milan J. Math.}, \textbf{57}(1) (1987), 275-285.
\bibitem{5} P. Terenzi, Equilateral sets in Banach spaces, \textit{Boll. Unione Mat. Ital. A(7)}, \textbf{13}(1) (1989), 119-124.
\bibitem{6} P. Wojtaszczyk, \textit{Banach spaces for analysts}, Cambridge Studies in Advanced Mathematics, Vol. 25, Cambridge University Press, Cambridge, 1981.
\end{thebibliography}
\end{document}